\documentclass[12pt]{article}
\usepackage{amsmath}
\usepackage{amsfonts}
\usepackage{amssymb}
\usepackage{graphicx,color}
\definecolor{lg}{rgb}{0.7,0.7,0.7}
\usepackage{color}

\makeindex
\newtheorem{Theorem}{\bf Theorem}%[subsection]}
\newtheorem{Lemma}[Theorem]{\bf Lemma}
\newtheorem{Proposition}[Theorem]{\bf Proposition}

\newtheorem{Remark}[Theorem]{\bf Remark}

\newtheorem{Corollary}[Theorem]{\bf Corollary}
\newtheorem{Observation}[Theorem]{\bf Observation}

\newcommand{\qed}{\hfill $\Box$ \bigskip}

\newcommand{\NN}{\mathbb N}

\newcommand{\CC}{\mathbb C}

\newcommand{\B}{\mathcal{B}}

\newcommand{\X}{\mathcal{X}}

\newcommand{\I}{\mathcal{I}}

\DeclareMathOperator{\s}{span}

\newcommand{\rk}{\mathrm{rank}\,}

\newcommand{\ix}{\mathcal{I}(\mathcal{X})}
\newcommand{\bx}{\mathcal{B}(\mathcal{X})}
\newcommand{\ixa}{\mathcal{I}^*(\mathcal{X})}
\newcommand{\im}{\mathrm{Im}\,}
\renewcommand{\ker}{\mathrm{Ker}\,}

\begin{document}
\title{Maps preserving the idempotency of Jordan products}

\author{
	Tatjana Petek\\ 
	\small \it University of Maribor, FERI, Smetanova ul. 17, SI-2000 Maribor, Slovenia\\
	\small \it IMFM, Jadranska ulica 19, 1000 Ljubljana, Slovenija\\
	\small \tt tatjana.petek@um.si
	\and 
	Gordana Radi\'{c}\\
	\small \it University of Maribor, FERI, Smetanova ul. 17, SI-2000 Maribor, Slovenia\\
	\small \it IMFM, Jadranska ulica 19, 1000 Ljubljana, Slovenija\\
	\small \tt gordana.radic@um.si
}

\date{}

\maketitle

\renewcommand{\thefootnote}{}



\renewcommand{\thefootnote}{\arabic{footnote}}
\setcounter{footnote}{0}

%%%%%%%%%%%%%%%%%%%%%%%%%%%%%
%%%%%%%%%%%% A B S T R A C T
%%%%%%%%%%%%%%%%%%%%%%%%%%%%
\begin{center}
	\textbf{Abstract}
\end{center}

Let $\bx$ be the algebra of all bounded linear operators on a complex Banach space $\X$  of dimension at least three. For an arbitrary nonzero $\lambda\in\CC$ we determine the form of  mappings  $\phi:\bx\to \bx$ with sufficiently large range such that $\lambda(AB+BA)$ is idempotent if and only if $\lambda(\phi(A)\phi(B)+\phi(B)\phi(A))$ is idempotent,   for all  $A$,$B\in \bx$. Note that $\phi$ is not assumed to be linear or additive.  

%%%%%%%%%%%%%%%%%%%%%%%%%%%
%%%%%%%%%%%% INTRODUCTION
%%%%%%%%%%%%%%%%%%%%%%%%%%%
\section{Introduction}

In the last few decades, there has been a lot of research interest in studying mappings on algebras, which leave invariant certain sets, relations and/or functions, often referred to as preservers. In this context, consider a map $\phi:{\bx}\rightarrow\bx$ such that for every pair $A,B\in\bx$ we have $\phi(A)\phi(B)=0$ whenever $AB=0$. These maps are typically called \textit{zero-product preserving maps}. The study of such maps has been extensive, mainly under the assumption that 
$\phi$ is linear (see \cite{Alaminos_Bresar} for a historical account). It is clear that characterizing these maps becomes significantly more challenging when neither linearity nor even additivity of $\phi$ is assumed. In this case, a stronger preserving property is required, leading to the condition $\phi(A)\phi(B)=0$ if and only if $AB=0$ for all $A, B\in\bx$. Maps satisfying this property are called \textit{preservers of zero-product in both directions}. In \cite{Bresar_Semrl}, the authors provide a detailed structure of all bijective maps on $\bx$ that preserve zero-product in both directions without assuming linearity. This result is relatively new, but a similar result concerning rank-one idempotents has been known for a longer time. Namely, in 2003, Moln\'{a}r characterized all bijective maps on rank-one idempotents that preserve the zero-product in both directions, which was a key step in his proof of Uhlhorn's Theorem, \cite{Molnar}.

Beyond zero-product-preserving maps, idempotent-preserving maps appear across multiple branches of mathematics. For instance, in mathematical morphology, many image-processing operations rely on idempotent transformations. Likewise, in operator theory and the study of $C^\ast$-algebras, preservers of idempotents under different types of products play a crucial role in analyzing morphisms and invariant subspaces. These contexts provide strong motivation for studying maps that preserve nonzero-idempotency. Specifically, a mapping $\phi:{\bx}\rightarrow\bx$ is said to \textit{preserve the (nonzero) idempotency of a product $\ast$} when $\phi(A)\ast\phi(B)$ is (nonzero) idempotent whenever $A\ast B$ is (nonzero) idempotent, for all $A,B\in\bx$. If the converse also holds --  the (nonzero) idempotency of $\phi(A)\ast\phi(B)$ implies the (nonzero) idempotency of $A\ast B$ -- then \textit{$\phi$ preserves (nonzero) idempotency of the product $\ast$ in both directions}. Restricting our attention to the standard product $A\ast B=AB$, the structure of surjective nonzero-idempotency preservers in both directions has been described in  \cite{Fang_Bai, Fang_Li_Pang}. A more general treatment, covering various types of products - including the usual product of $k$ operators, $k\geq2$, and the Jordan triple product $A\ast B=ABA$ - can be found in \cite{Petek}.

In this paper, we focus on the Jordan product of $A, B\in\bx$ defined as $A\ast B=AB+BA$. Additive zero-Jordan product preservers have been studied in \cite{Chebotar_Ke_Lee, Chebotar_Ke_Lee_Zhang, Hou_Jiao, Qi_Wang, Zhao_Hou}, while linear preservers of nonzero idempotency of the Jordan product were explored in \cite{Fang}. Let us note that the Jordan product of operators $A,B\in\bx$ is sometimes in the literature also defined as $A\circ B=\frac{1}{2}\left(AB+BA\right)$,
which has the desirable property that $I\circ I=I$. In the sequel, whenever we refer to Jordan product, we have the product $\circ$ in mind. 
Results concerning additive maps preserving nonzero idempotency of the Jordan product $\circ$ can be found in \cite{Taghavi_Hosseinzadeh_Rohi}. 

A natural question arises: can meaningful results still be obtained if the assumptions of linearity and additivity are removed? In this work, we address an even more challenging problem by setting the preserver property in a more general form: the Jordan product of $A$ and $B\in\bx$ is idempotent if and only if the Jordan product of $\phi(A)$ and $\phi(B)$ is idempotent. Notably,  we allow the possibility that a nonzero Jordan product between $A$ and $B$ corresponds to a zero-Jordan product between $\phi(A)$ and $\phi(B)$, and vice-versa. As far as we know, such a preserving property has not yet been studied in the literature. 

Throughout this paper, if not stated otherwise, $\X$ will denote a complex Banach space of dimension at least three, and $\bx$ will stand for the algebra of all bounded linear operators on $\X$. By ${\X}'$ we denote the topological dual of $\X$. We adopt standard notation and terminology concerning idempotents. Specifically, an operator $P\in\bx$ is idempotent if $P^2=P$ and anti-idempotent if $P^2=-P$ and  $P\neq0$. We denote by ${\ix}$ and $-\ix\subseteq\B(\X)$ the set of all idempotent and anti-idempotent operators, respectively. Let
$\mathcal{I}^\pm(\X)=(\ix\backslash\{0\})\cup-\ix$. The symbols $\ixa$,  ${\I}_1(\X)$ and $\mathcal{I}_1^\pm(\X)$  will represent the subsets of nonzero idempotents, rank-one idempotents and the union of rank-one idempotents and rank-one anti-idempotents, respectively.

We say that an idempotent $P$ is nontrivial if it is neither equal to zero nor to the identity operator on $\X$, denoted throughout by $I$. Analogously, an anti-idempotent $Q$ is nontrivial if $Q\neq-I$ (by definition $Q\neq0$). 

In the following, we will refer to a natural relation on idempotents called orthogonality. We say that $P,Q\in\ix$ are orthogonal, denoted by $P\bot Q$, if $PQ=QP=0$. We will also use the notation ${\CC}^\ast={\CC}\backslash\{0\}$ and $\mathcal{F}_k(\X)\subseteq\bx$ to denote the set of all operators of rank at most $k$, where $k\in\NN$. %Here is our first main result.

\begin{Theorem}\label{main_result}
Let $\X$ be an infinite-dimensional complex Banach space and let $\phi:{\bx}\rightarrow\bx$ be a map whose range contains the set $\mathcal{I}^\pm(\X)$. Suppose that
$$
A\circ B\in{\I(\X)}\quad\Longleftrightarrow\quad\phi(A)\circ\phi(B)\in\I(\X),
$$ 
for every $A,B\in\bx$. Then either there exists a bounded invertible linear or conjugate-linear operator $T:{\X}\rightarrow\X$ and  $\lambda\in\left\{-1,1\right\}$ such that
$$
\phi(X)=\lambda TXT^{-1},\qquad\text{ for every }X\in{\mathcal B}({\X}),
$$
or, only if $\X$ is reflexive,  there exists a bounded invertible linear or conjugate-linear operator $T:{\X'}\rightarrow \X$ and  $\lambda\in\left\{-1,1\right\}$  such that
$$
\phi(X)=\lambda TX'T^{-1},\qquad\text{ for every }X\in{\mathcal B}({\X}),
$$
where $X'$ denotes the adjoint operator of the operator $X$.
\end{Theorem}

When $\X$ is a finite-dimensional complex vector space, then it will be regarded as ${\CC}^n$ for some $n\in\NN$, and its members are seen as column vectors. Linear endomorphisms of ${\CC}^n$ will be identified with $n\times n$ complex matrices, denoted by $M_n$.
By ${\mathcal I}_n$, ${\mathcal I}^*_n$, ${\mathcal I}_{n,1}$, $\mathcal{I}_n^\pm\subseteq M_n$ we denote the set of idempotent, the set of nonzero idempotent, the set of rank-one idempotent, and the union of rank-one idempotent and rank-one anti-idempotent matrices, respectively.

\begin{Theorem}\label{main_result_matrix}
Let $\phi:M_n\rightarrow M_n$, $n\geq3$, be a map whose range contains the set $\mathcal{I}_n^\pm$. Suppose that
$$A\circ B\in{\mathcal I}_{n}\quad\Longleftrightarrow\quad\phi(A)\circ\phi(B)\in{\mathcal I}_n,$$
for every $A,B\in M_n$. Then there exists a field automorphism $\sigma$ of $\CC$, a non-singular matrix $T\in M_n$ and $\lambda\in\{1,-1\}$ such that
$$
\phi([x_{ij}])=\lambda T [\sigma(x_{ij})]^\diamond T^{-1},\qquad\text{for every }[x_{ij}]\in M_n,
$$
where $\diamond$ denotes one of the maps on $M_n$: identity or transposition.
\end{Theorem} 

\begin{Remark}
By \cite{Kestelman} there exist many non-continuous automorphisms of the complex field. 
\end{Remark}

Applying Theorems \ref{main_result} and \ref{main_result_matrix}, we obtain the following characterizations as a simple consequence. Here, for a scalar $\alpha\in\mathbb C$  the notations $\alpha\mathcal{I}^\pm(\mathcal{X})$ and $\alpha\mathcal{I}_n^\pm$  refer to the sets of all operators of the form $A=\alpha P$, where $P\in\mathcal{I}^\pm(\X)$ or $P\in\mathcal{I}_n^\pm$, respectively.

\begin{Theorem}\label{main_result2}
Let $\mathcal{X}$ be an infinite-dimensional complex Banach space and let $\alpha\in\mathbb{C}$ be nonzero. Suppose that $\phi:{\bx}\rightarrow\bx$ is a map whose range contains the set $\sqrt{(2\alpha)^{-1}}\;\mathcal{I}^\pm(\mathcal{X})$ and that
$$\alpha(AB+BA)\in{\I(\X)}\quad\Longleftrightarrow\quad\alpha(\phi(A)\phi(B)+\phi(B)\phi(A))\in\I(\X),$$ 
for every $A,B\in\bx$. Then either there exists a bounded invertible linear or conjugate-linear operator $T:{\X}\rightarrow\X$ and $\lambda\in\left\{-1,1\right\}$ such that
$$
\phi(X)=\lambda TXT^{-1},\qquad\text{ for every }X\in{\mathcal B}({\X}),
$$
or, only if $\X$ is reflexive,  there exists a bounded invertible linear or conjugate-linear operator $T:{\X'}\rightarrow \X$ and $\lambda\in\left\{-1,1\right\}$  such that
$$
\phi(X)=\lambda TX'T^{-1},\qquad\text{ for every }X\in{\mathcal B}({\X}),
$$
where $X'$ denotes the adjoint operator of the operator $X$.
\end{Theorem}

\begin{Theorem}\label{main_result2_matrix}
Let $\alpha\in\CC$ be nonzero, and let $\phi:M_n\rightarrow M_n$, $n\geq3$, be a map whose range contains $\sqrt{(2\alpha)^{-1}}\;\mathcal{I}_n^\pm$.  Suppose that
$$\alpha(AB+BA)\in{\I_n}\quad\Longleftrightarrow\quad\alpha(\phi(A)\phi(B)+\phi(B)\phi(A))\in\I_n,$$ 
for every $A,B\in M_n$. Then there exists a field automorphism $\sigma$ of $\CC$, a non-singular matrix $T\in M_n$ and $c_\alpha$, $c_\alpha^2=\sigma(\alpha)/\alpha$, such that
	$$
	\phi([x_{ij}])=c_\alpha T [\sigma(x_{ij})]^\diamond T^{-1},\qquad\text{for every }[x_{ij}]\in M_n,
	$$ 
where $\diamond$ denotes one of the maps on $M_n$: identity or transposition.
\end{Theorem}

In the next section, we establish some preliminary results, followed by the proofs of the above theorems in the subsequent section.

%%%%%%%%%%%%%%%%%%%%%%%%%%%%%%%%%%%
%%%%%%%%%%%%%%%%%%%%%%%%%%%%%%%%%%%
\section{Preliminaries}

Recall that every rank-one operator on $\X$ can be written as $x\otimes f$ for some nonzero vector $x\in\X$ and some nonzero functional $f\in\X'$. It is defined by $(x\otimes f)z=f(z) x$ for every $z\in\X$, and it holds $A(x\otimes f)=Ax\otimes f$ and $(x\otimes f)A=x\otimes A'f$ for every $A\in{\mathcal{B}\left(X\right)}$, where $A'$ stands for the adjoint operator of $A$.  The rank-one operator $x\otimes f$ is idempotent (anti-idempotent) if and only if $f(x)=1$ ($f(x)=-1$, resp.). Operator $x\otimes f$ is nilpotent if and only if $f(x)=0$. It is  clear that  $\lambda x\otimes f=x\otimes \lambda f$ for every scalar $\lambda$. \\

First, we will introduce some elementary results that will be the cornerstone of the subsequent investigation. Let us start with two simple observations stated without proofs. 

\begin{Observation}\label{Observation}
Let $A$ and $X\in \bx$ be nonzero operators. Then the following statements hold true.
\begin{enumerate}
    \item[(i)]  $A\circ X \in \I(\X)$ and $-A\circ X \notin \I(\X)$ if and only if $A\circ X \in \I^\ast(\X)$.
    \item[(ii)] $A\circ X \in \I(\X)$ and $-A\circ X \in \I(\X)$ if and only if $A\circ X =0$.
\end{enumerate}
\end{Observation}

We now proceed with several auxiliary lemmas that will be instrumental in the sequel.

%%%%%%%%%%%%%%%%%%% Jordan product between A in rank-one operator
\begin{Lemma}\cite[Lemma 2.4]{Fang}\label{f(x)neq0}
Let $A$ and $x\otimes f\in\B(\X)$ be nonzero operators. If $f(x)\neq0$, then $A\circ x\otimes f$ is a nonzero idempotent if and only if $Ax=\frac{1}{f(x)}x$ or $A'f=\frac{1}{f(x)}f$.
\end{Lemma}

\begin{Lemma}\label{f(x)neq0_1}
Let $A$ and $x\otimes f\in\B(\X)$ be nonzero operators. If $f(x)\neq0$, then $A\circ x\otimes f=0$ if and only if $Ax=0$ and $A'f=0$.
\end{Lemma}

\begin{proof}
Let us prove only the nontrivial part. Suppose that $A\circ x\otimes f=0$ for some $x\in\X$ and $f\in\X'$ with $f(x)\neq0$. Then, $Ax=\lambda x$ and $A'f=-\lambda f$ for some scalar $\lambda$. As $f(x)\neq0$, from $f(Ax)=\lambda f(x)$ and $f(Ax)=-\lambda f(x)$ it follows that $\lambda=0$.
\qed\end{proof}

\begin{Lemma}\label{f(x)=0}
Let $A\in\B(\X)$ and $x\otimes f\in\B(\X)$ be nonzero operators. If $f(x)=0$, then $A\circ x\otimes f$ is a nonzero idempotent if and only if $f(Ax)=2$ and $f(A^2x)=0$.
\end{Lemma}

\begin{proof}
Again, the `if' statement is obvious, so let us prove the `only if'. Direct calculation of $A\circ x\otimes f\in\I(\X)$ gives
$$
2Ax\otimes f+2x\otimes A'f=f(Ax)Ax\otimes f+f(x)Ax\otimes A'f+f(A^2x)x\otimes f+f(Ax)x\otimes A'f,
$$
which is the same as
$$
Ax\otimes\left(\left(2-f(Ax)\right)f-f(x)A'f\right)+x\otimes\left(\left(2-f(Ax)\right)A'f-f(A^2x)f\right)=0.
$$ 
Then apply $f(x)=0$ to see that
\begin{equation}\label{prop:f(x)=0}
    \left(2-f(Ax)\right)Ax\otimes f+x\otimes\left((2-f(Ax))A'f-f(A^2x)f\right)=0.
\end{equation}
If $x$ and $Ax$ are linearly independent, we have $2-f(Ax)=0$ and $f(A^2x)=0$. If this is not the case, $Ax=\lambda x$ for some scalar $\lambda$ and so, $f(Ax)=0$ and $f(A^2x)=0$. But then we can rewrite the identity (\ref{prop:f(x)=0}) to $x\otimes(2\lambda f+2A'f)=0$.  This gives  $A\circ(x\otimes f)=0$, a contradiction.
\qed\end{proof}

\begin{Lemma}\label{lem:lambda_x}
Let $A\in\bx$ be a nonzero operator. Suppose that for some $\lambda \in \CC\backslash\{0\}$ and some nonzero $x\in\X$ we have $A\circ x\otimes f \in \I^\ast(\X)$, for all $f\in \X'$ with $f(x)=\frac{1}{\lambda}$. Then $Ax=\lambda x$.
\end{Lemma}

\begin{proof}
Assuming $A\circ x\otimes f\in \I^\ast(\X)$ and applying Lemma \ref{f(x)neq0} gives that $Ax=\lambda x$ or $A'f=\lambda f$. If $Ax\neq \lambda x$, then $A'f=\lambda f$ for every $f$ with $f(x)=\frac1\lambda$. Replacing $f$ by $f+g$, where $g\in \X'$ satisfies $g(x)=0$, one obtains that $g(Ax)=0$ yielding that $x$ and $Ax$ are linearly dependent, say $Ax=\mu x$ for some $\mu \neq \lambda$. From $A'f=\lambda f$ it follows that $\mu = \lambda$, a contradiction. Therefore, $Ax=\lambda x$.
\qed\end{proof}

%%%%%%%%%%%%%%%%%% Characterization of 0 and I
We continue with some characterizations of trivial idempotent operators in terms of the idempotency of the Jordan product $\circ$.

\begin{Lemma}\label{Lemma:A=0}
Let $A\in\B(\X)$. Then the following statements are equivalent.
\begin{enumerate}
    \item[(i)] $A=0$.
    \item[(ii)] $A\circ X\in\I(\X)$ for every $X\in\B(\X)$.
    \item[(iii)] $A\circ R\in\I(\X)$ for every $R\in\I_1(\X)\cup (-\I_1(\X))$.
\end{enumerate}
\end{Lemma}

\begin{proof}
Clearly (i)$\Rightarrow$(ii)$\Rightarrow$(iii). Assume that (iii) holds true. By Observation \ref{Observation}, $A\circ x\otimes f =0$ for all $x\in\X$ and all bounded functionals $f$ with $f(x)=1$. Then, by Lemma \ref{f(x)neq0_1}, $Ax=0$ for every $x\in \mathcal{X}$.
\qed\end{proof}

\begin{Lemma}\label{Lemma:A=I}
Let $A\in\B(\X)$. Then the following statements are equivalent.
\begin{enumerate}
    \item[(i)] $A=I$.
    \item[(ii)]$A\circ R\in\ixa$ for every $R\in{\I}_1(\X)$.
\end{enumerate} 
\end{Lemma}

\begin{proof} 
While (i)$\Rightarrow$(ii) is clear, the reverse implication is a simple consequence of Lemma \ref{lem:lambda_x}.
\qed\end{proof}

If $P\in\B(\X)$ is a nontrivial idempotent, the underlying Banach space $\X$ can be decomposed into the direct sum ${\X}=\im P\oplus\ker P$, where $\im P$ and $\ker P$ denote the image and the kernel of $P$, respectively. 
 
The next Lemma is then technical.

\begin{Lemma}\label{Lemma:id}
Let $\mathcal X$ be a Banach space with $\dim\X\ge2$ and let $A\in {\B}({\X})\backslash \{0,I,-I\}$. If $A^2=0$ or $A^2=I$, then there exists an operator $B\in \B(\X)$ of rank  at most two such that $A\circ B\in \ix$ and $B\circ B\notin\ix$.  
\end{Lemma}

\begin{proof}
Let us firstly assume that $A^2=0$ and $A\neq 0$. Then there exists an $x\in \X$ such that $x$ and $Ax$ are linearly independent. Hence, we can take an $f\in {\X}'$ such that $f(x)=0$ and $f(Ax)=2$. By setting $B:=x\otimes A'f-Ax\otimes f +x\otimes f$, we compute that $A\circ B=\frac{1}{2}Ax\otimes f +\frac{1}{2}x\otimes A'f$, which is a sum of two orthogonal rank-one idempotents, and thus an idempotent. Moreover, since $Bx=2x$, $B^2$ cannot be an idempotent.
	
It remains to consider  $A^2=I$. The operator $A$ is a nonscalar involution, so $A=I-2Q$ for some nontrivial idempotent $Q$. As $Q$ is nontrivial, we can choose a nonzero $y\in \ker Q$ and a nonzero $z\in \im Q$ to have $Ay=y$ and $Az=-z$. Then we pick a functional $g\in{\X}'$ such that $g(\ker Q)=0$ and $g(z)=2$ and take a $h\in{\X}'$ with $h(y) = 1$ and $h(\im Q)=0$. It is not difficult to verify that $A'g=-g$ and $A'h=h$. Set $B:=y\otimes g+z\otimes h$, then a simple computation shows that $A\circ B=0$. Furthermore, it is easy to see that $B^2z=2z$ and so, $B^2=B\circ B$ is not idempotent.
\qed\end{proof}

%%%%%%%%%%%%%%%%%%% P=Q
Up to now,  we have considered those operators $A\in\bx$ for which  $A$ or $A^2$ is a trivial idempotent.  We now turn our attention to the case of nontrivial idempotents.

\begin{Lemma}\label{Lemma:PQ=0}
Let  $A, P\in\bx$ with $A$ nonzero and $P$ nontrivial idempotent. If $A\circ(I-P)=0$, then $\im P$ is a nontrivial $A$-invariant subspace, and  $Ax=0$ for every $x\in\ker P$.
	
In particular, if $\rk P=1$, then the assumption $A\circ (I-P)=0$ gives that $A=\lambda P$, for some nonzero scalar $\lambda$, and additionally, when $A$ is idempotent, we have $A=P$.
\end{Lemma}

\begin{proof}
With respect to the decomposition ${\X}=\im P\oplus\ker P$, operators $A$ and $P$ have the following matrix representations:
$$
A=\begin{bmatrix}A_{11}&A_{12}\\A_{21}&A_{22}\end{bmatrix}
\qquad\text{and}\qquad 
P=\begin{bmatrix}I&0\\0&0\end{bmatrix}.
$$
As $A\circ (I-P)=0$ it is clear that $A_{12}=0$, $A_{21}=0$ and $A_{22}=0$.  The second claim is now obvious.
\qed\end{proof}

\begin{Lemma}\label{Lemma:-P}
Let $P,Q\in{\I}_1(\X)$. If $P\neq Q$ and $Q\circ (I-P)\in \mathcal{I}(\mathcal{X})$, then there exists a rank-one idempotent $R\in{\I}_1(\X)$ such that $Q\circ R=0$ and $P\circ R\notin \mathcal{I}(\mathcal{X})$.
\end{Lemma}

\begin{proof}
Let $P:=x_1\otimes f_1$ and $Q:=x_2\otimes f_2$, for some $x_1,x_2\in\X$ and $f_1,f_2\in{\X}'$ with $f_1(x_1)=1=f_2(x_2)$. The condition $Q\circ (I-P)\in \mathcal{I}(\mathcal{X})$ forces us to assume that $Q\circ (I-P)\in\ixa$, since otherwise,   $Q=P$ by virtue of Lemma \ref{Lemma:PQ=0}. Note that the Jordan product is a commutative operation, thus, $Q\circ(I-P)=(I-P)\circ(x_2\otimes f_2)\in\ixa$. Clearly, $(I-P)x_2=x_2$ or $(I-P)'f_2=f_2$ by Lemma \ref{f(x)neq0}, which is equivalent to $Px_2=0$ or $P' f_2 = 0$. In particular, $f_1(x_2)=0$ or $f_2(x_1)=0$.

Suppose $f_1(x_2)=0$. Vectors $x_1$ and $x_2$ are obviously linearly independent, so we can choose a nonzero functional $g\in{\X'}$ with $g(x_1)=0=g(x_2)$.  It is easy to see that $f_1,f_2$ and $g$ are now linearly independent, and we can pick a vector $y\in\X$ such that $f_1(y)=0=f_2(y)$ and $g(y)=1$. Set $R:=\frac{1}{2}(y+x_1-f_2(x_1)x_2)\otimes (g+f_1)$ and check that $R\in  {\mathcal I}_1(\mathcal{X})$ and that $Q\circ R=0$. On the other hand, $P\circ R$ is not idempotent since its trace equals $\frac{1}{2}$.
	
The case  $f_2(x_1)=0$ can be treated in a very similar manner;  there exists a vector $y \in \X$ such that $f_1(y)=0=f_2(y)$ and a functional $g \in \X'$ with $g(x_1)=0=g(x_2)$ and $g(y)=1$. Define $R:=\frac12(y+x_1)\otimes(g+f_1-f_1(x_2)f_2)$ which again satisfies $R \in \mathcal{I}_1(\X)$ and $Q \circ R = 0$, while $P \circ R$ is not idempotent.
\qed\end{proof}

\begin{Lemma}\label{Lemma:PQ}
Let $P,Q\in\ix$ be nontrivial idempotents. Then the following statements are equivalent.
\begin{enumerate}
    \item[(i)] $P\bot Q$.
    \item[(ii)] $P\circ Q=0$.
    \item[(iii)] $-P\circ Q\in \ix$.
\end{enumerate}
\end{Lemma}

\begin{proof}
The implications (i)$\Rightarrow$(ii)$\Rightarrow$(iii) are straightforward. It remains to prove (iii)$\Rightarrow$(i). So, let us assume that $-P\circ Q\in \ix$. With respect to the direct sum decomposition ${\X}=\im P\oplus\ker P$, operators $P$ and $Q$ have the following matrix representations:
$$
P=\begin{bmatrix}I&0\\0&0\end{bmatrix}
\qquad \text{and}\qquad 
Q=\begin{bmatrix}A&B\\C&D\end{bmatrix}.
$$
That $Q$ is idempotent and the property $-P\circ Q=-\frac{1}{2}\left[\begin{smallmatrix}2A&B\\C&0 \end{smallmatrix}\right]\in\ix$ can be presented as
\begin{align}
	&A=A^2+BC, &\text{and}&\hspace{2cm}4A=-4A^2-BC\label{eq:A} \\
	&B=AB+BD,&\text{and}&\hspace{2cm}\ B=-AB,\label{eq:B}\\
	&C=CA+DC&\text{and}&\hspace{2cm}\ C=-CA, \label{eq:C}\\
	&D=CB+D^2&\text{and}&\hspace{2.1cm}\ 0=CB. \label{eq:D}
\end{align}
From \eqref{eq:D}, we observe that $D$ must be idempotent and from \eqref{eq:C} that $DC=C-CA=2C$, which is the same as $\left(D-2I\right)C=0$. Since $D$ is idempotent, it is clear that $D-2I$ is invertible; thus,  $C=0$. That $B=0$ follows similarly from \eqref{eq:B}. Then, from \eqref{eq:A}, we easily infer that $A=0$, and the proof is closed.
\qed\end{proof}

Despite its simplicity, the following lemma serves as a powerful tool that facilitates a unified approach to several subsequent results.

\begin{Lemma} \label{lem:X}
Let $\mathcal{U}$ be a Banach space with $\dim\;\mathcal U\ge2$ and suppose that $\mathcal{U}=\mathcal{Y} \oplus \mathcal{Z}$, where $\mathcal{Y}$ is a $k$-dimensional subspace of  $\mathcal{U}$ with $1\le k <\dim \mathcal{U}$. Let operators $A$ and $T_R$ be represented in the operator matrix form with respect to this decomposition by
$$A=
\begin{bmatrix}
A_{11} & A_{12}\\ 
A_{21} & A_{22}
\end{bmatrix}
\qquad\text{and}\qquad
T_R=
\begin{bmatrix}
X & R\\  
0 & 0
\end{bmatrix},
$$
where $0\neq X\in\mathcal{B(Y)}$ and $R: \mathcal{Z}\to  \mathcal{Y}$ is a linear rank-at-most-one   operator. If $A\circ T_R$ is a nonzero idempotent for every  $R$, then $A_{21}=0$ and $A_{11} \circ  X $ is nonzero idempotent.
\end{Lemma}

\begin{proof}
 From
$$
A\circ T_R =
\begin{bmatrix}
A_{11}\circ X + \frac{1}{2}RA_{21} &  \frac{1}{2}(A_{11}R+XA_{12}+RA_{22})\\
\frac{1}{2}A_{21}X &  \frac{1}{2}A_{21}R
\end{bmatrix}\in\ixa 
$$
we observe that $A\circ T_R$ is a nonzero idempotent of rank at most $2k$. So, for every $R$, the trace of $A\circ T_R$ belongs to the discrete set
$\{1,2,\dots ,2k\}$. On the other hand,
\begin{center}
    $\mathrm{trace} (A\circ T_R)= \mathrm{trace}(A_{11}\circ X) + \frac12 \mathrm{trace}(R A_{21})+\frac12\mathrm{trace}(A_{21}R).$	    
\end{center}
Since $R$ is arbitrary, we easily get $A_{21}=0$, and consequently, $A_{11}\circ X$ must be a nonzero idempotent as desired.
\qed\end{proof}

\begin{Corollary}\label{corr:J1}
Let $A,B\in\bx$  and let $\lambda\in\CC$ be nonzero. Suppose that $A \circ X\in  {\ixa}$ if and only if $B \circ X \in {\ixa}$ for every $\lambda X\in {\mathcal I}_1(\X)$. Then, for every nonzero $x\in\X$ we have $Ax=\lambda x$ if and only if  $Bx=\lambda x$.
\end{Corollary}

\begin{proof}
 The claim is a direct consequence of  Lemma \ref{lem:X} in the case $k=1$; alternatively, it can also be derived from Lemma \ref{lem:lambda_x}.
\end{proof}\qed

Let $A\in\bx$ be an algebraic operator of order $n\in\NN$. Then there exists a polynomial $p(x)\in{\CC}[x]$ of degree $n$ such that $p(A)=0$, and $q(A)\neq0$ for any polynomial $q\in{\CC}[x]$ of degree less than $n$. It is well known that any
algebraic operator $A\in\bx$ shares the same finite spectrum with its adjoint $A'$, and that each spectral value of $A$ and $A'$ is in fact an eigenvalue.

\begin{Lemma}\label{corr:alg}
Let $A,B\in \bx$ be algebraic operators such  that 
$$
A \circ X\in  {\ixa}\Leftrightarrow  B \circ X \in {\ixa}
\qquad\text{and}\qquad
A\circ X=0\Leftrightarrow B\circ X=0
$$
for every rank-one $X\in\bx$. Then, for every nonzero $x\in\X$ we have $Ax=0$ if and only if  $Bx=0$.    
\end{Lemma}

\begin{proof}
Let us assume  that $Ax=0$ for some nonzero $x\in\X$. Since $0$ is also an eigenvalue of $A'$, there exists a nonzero $f\in\X'$ such that $A'f = 0$. This gives rise to $A\circ x\otimes f = 0$, and in turn, $B\circ x\otimes f = 0$. Then $Bx = \lambda x$ and $B'f = -\lambda f$ for some $\lambda\in\CC$. If $\lambda\neq0$, then, by Corollary \ref{corr:J1}, $Ax = \lambda x$, a contradiction. Therefore, $\lambda=0$ and $Bx = 0$.  
\qed\end{proof}

A special class of algebraic operators is the set of tripotent operators, denoted by ${\mathcal K}(\X) \subseteq \bx$. An operator $A \in \bx$ is called tripotent if $A^3 = A$. According to \cite[Proposition 1]{tripotent}, for every $A \in {\mathcal K}(\X)$ there exist unique orthogonal  $P, Q \in \ix$ such that $A = P - Q$. In particular, there exists a decomposition of $\X$ with respect to which the operator $A \in {\mathcal K}(\X)$ can be represented in operator matrix form as $A = I \oplus -I \oplus 0$,  where some of the summands may be missing if the corresponding subspaces are trivial.

\begin{Corollary}\label{corr:D}
Let $A,B\in \bx$ be such that
$$
A \circ X\in  {\ixa}\Leftrightarrow  B \circ X \in {\ixa}
\qquad\text{and}\qquad
A\circ X=0\Leftrightarrow B\circ X=0
$$
for every  $X\in {\mathcal I}_1^\pm(\X)$. If at least one of the operators $A$ or $B$ is tripotent, then $A=B$.
\end{Corollary}

\begin{proof}
Let us suppose that $A\in\mathcal{K}(\X)$. Then there exists a decomposition of $\X$ such that, with respect to a suitable basis, the operator $A$ admits the matrix representation $A=I\oplus-I\oplus0$. By combining Corollary \ref{corr:J1} and Lemma \ref{corr:alg}, we conclude that $B=A$.
\end{proof}\qed

By $J_k(\lambda)$ we denote a $k$-by-$k$ upper triangular Jordan block with scalar $\lambda$ on the main diagonal and scalar $1$ on the upper diagonal when $k>1$; all other entries are zero.

\begin{Lemma}\label{lem:Jk}  
Assume $n\geq2$ and let $A\in M_n$ be a matrix with only one nonzero spectral point. Suppose  that for some nonzero $\lambda \in\CC$ the matrix	$A \circ T_{k,R}$ is a nonzero idempotent for every matrix 
\begin{equation}\label{eq:Tj}
T_{k,R} = 
\begin{bmatrix}
J_k(\lambda)^{-1} & R \\
0 & 0
\end{bmatrix}\in M_n, \qquad k=1,2,\dots, n-1,
\end{equation}
where $R$ is any at most rank-one matrix. Then, if $e_k$ denote the standard basis vector in ${\CC}^n$ with a single nonzero entry $1$ in the $k$-th position, we have $Ae_k=J_n(\lambda)e_k$, $k=1,2,\dots, n-1$.	
\end{Lemma}

\begin{proof}
We apply induction on $k$, so let us begin with $k=1$. By Lemma \ref{lem:X} we get that $Ae_1=\lambda e_1=J_1(\lambda)e_1$. We are done if $n=2$. 
	
Assume now that $n>2$, and suppose the assertion holds true for some $1\le k<n-1$. By the inductive hypothesis, the matrix $A$ can be written as
$$
A=\begin{bmatrix}
A_{11}&A_{12}\\
A_{21}&A_{22}
\end{bmatrix},
\qquad\text{where}\qquad 
A_{11}=\begin{bmatrix}
    J_k(\lambda)&\ast\\
    0&\ast
\end{bmatrix}\in M_{k+1}.
$$
Our goal is to show that $A_{11}=J_{k+1}(\lambda)$ and $A_{21}=0$. To do this, apply Lemma \ref{lem:X} with $T_{k+1,R}$ in place of $T_R$. It follows that  $A_{21}=0$ and that $A_{11}\circ J_{k+1}(\lambda)^{-1}$ is a nonzero idempotent. Since $A$ is assumed to have a single spectral point, we obtain that $A_{11}$ is an upper-triangular matrix with all diagonal entries equal to $\lambda$. 

 Now, both $A_{11}$ and $J_{k+1}(\lambda)^{-1}$ are upper-triangular matrices with a single nonzero eigenvalue, hence $A_{11}\circ J_{k+1}(\lambda)^{-1}=I_{k+1}$. This equation can be rewritten as a Sylvester equation
$
A_{11} J_{k+1}(\lambda)^{-1}+J_{k+1}(\lambda)^{-1}A_{11}=2I_{k+1},
$
where $A_{11}$ is the unknown matrix. Since the spectra of  $J_{k+1}(\lambda)^{-1}$ and $-J_{k+1}(\lambda)^{-1}$ are disjoint, Theorem 2.4.4.1 from \cite{MA} ensures that this Sylvester equation has a unique solution. Therefore, $A_{11}=J_{k+1}(\lambda)$, completing the inductive step.
\end{proof}\qed

Let us introduce another algebraic family of operators, denoted by ${\mathcal T}_k(\lambda)\subseteq\bx$, $k<\dim\X$, consisting of all operators $A$ such that there exists a direct sum decomposition of $\X$ with respect to which $A$ can be represented in operator matrix form as $A=SJ_k(\lambda)S^{-1}\oplus0$ for some invertible $S\in M_k$. Further, for $n<\dim\X$ let  
\begin{equation}\label{T(n,lambda)}
\mathcal{T}(n;\lambda) := \bigcup_{k=1}^{n} \mathcal{T}_k(\lambda).    
\end{equation}

\begin{Lemma}\label{lem:Jn(lambda)}
Let ${\X}={\mathcal Y}\oplus\mathcal Z$, where $\mathcal Y$ is an $n$-dimensional subspace of $\X$ with $1\leq n<\dim\X$, and let $\lambda\in\mathbb{C}^\ast$. Suppose that $A,B\in\bx$ are algebraic operators such that
$$
A \circ X\in  {\mathcal I}^*({\X})\Leftrightarrow  B \circ X \in {\mathcal I}^*({\X})
\qquad\text{and}\qquad
A\circ X=0\Leftrightarrow B\circ X=0,
$$
for every $X \in \mathcal{F}_1({\X})\cup {\mathcal{T}}(n-1;\lambda^{-1})$. If with respect to the decomposition $\X={\mathcal Y}\oplus\mathcal Z$ we can represent the operator $A$ in an operator-matrix form as
$$
A=\begin{bmatrix}
J_n(\lambda)&A_{12}\\
0&A_{22}
\end{bmatrix},
$$
then $By=Ay$ for every $y\in \mathcal Y$.	
\end{Lemma}

\begin{proof}
According to the decomposition ${\X}={\mathcal Y}\oplus\mathcal Z$, the operator $B$ can be written as 
$
B=\left[\begin{smallmatrix}
    B_{11} & B_{12} \\
    B_{21}& B_{22}
\end{smallmatrix}\right].
$
We aim to see that $B_{21}=0$ and $B_{11}=J_n(\lambda)$. To achieve this, consider the operators 
$$
T_{k,R}=\begin{bmatrix}
J_k(\lambda)^{-1} & R\\
0&0
\end{bmatrix}, \qquad\text{for } k=1,2,\ldots,n-1,
$$
where $R:{\mathcal Z}\rightarrow\mathcal Y$ is an arbitrary linear operator of rank at most one. Note that $T_{k,R}\in \mathcal{T}_k(\lambda^{-1})$ for every such $R$. Since $A\circ T_{k,R}\in \mathcal{I}^\ast(\X)$ for all such $R$, it follows from our assumption that also  $B\circ T_{k,R}\in \mathcal{I}^\ast(\X)$ for every rank-at-most-one $R$. Applying Lemma \ref{lem:X} we then have $B_{21}=0$ and $B_{11}\circ J_{k}(\lambda)^{-1}\in \mathcal{I}^\ast(\X)$. Since the only eigenvalue of $B_{11}$ is $\lambda$, otherwise we are in a contradiction with Corollary \ref{corr:J1} or Lemma \ref{corr:alg}, it follows that $B_{11}=J_n(\lambda)$ from Lemma \ref{lem:Jk}.
\qed\end{proof}

We now turn our attention to operators that can be represented as operator matrices $J_n(0)\oplus 0$, with respect to some basis in $\X$. In this context, nilpotent operators of rank one will play a significant role. Thus, we continue this section with a result that specifically addresses such operators in general. Note that the set ${\mathcal N}_1({\X})\subseteq\B(X)$ consists of all nilpotent operators of rank one, and by applying Lemma \ref{f(x)=0}, we deduce that for every $N\in {\mathcal N}_1(\X)$ we have $N\circ x\otimes f\in\ixa$ if and only if  $f(Nx)=2$.

\begin{Proposition}\label{prop:1234}
Let $A$, $B\in\bx$ be operators satisfying
$A \circ N\in  {\ixa}$ if and only if $B \circ N \in {\ixa}$, for every $N\in{\mathcal N}_1(\X)$. Then the following statements hold true.
\begin{enumerate}
    \item[(i)] 	For any $y\in \X$, such that the vectors $y$, $Ay$, $A^2y$  are linearly independent or, $y$, $Ay$ are linearly independent and $A^2y=0$, there exist scalars $\alpha$ and $\gamma$ (dependent on $y$) such that
        \begin{equation}\label{eq:12}
			By = \alpha y +Ay +\gamma A^2 y
	\end{equation}
    and $B^2y \in {\s}\left\{y,A^2y\right\}$.
    \item[(ii)] For any $x\in\X$ such that $x$, $Ax$, $A^2x$ and $A^3x$ are linearly independent or $x$, $Ax$, $A^2x$  are linearly independent and $A^3x=0$, there are sclalars $\alpha$ and $\gamma$ such that 
	\begin{equation}\label{eq:23}
		\begin{aligned} 
				Bx &= \alpha x +Ax +\gamma A^2 x, \\
				BAx &= \alpha Ax + A^2x + \gamma A^3 x,
		\end{aligned}
	\end{equation}
    and consequently, $BAx = ABx$. 
    \item[(iii)] For any $x\in\X$ such that $x$, $Ax$, $A^2x$, $A^3x$ and $A^4x$ are linearly independent, or, $x$, $Ax$, $A^2x$, $A^3x$ linearly independent and $A^4x=0$ we have $Bx=Ax$, $BAx=A^2x$ and $BA^2x=A^3x$.
\end{enumerate}
\end{Proposition}

\begin{proof}
(i) We firstly observe that  $By$ and $B^2y \in {\s}\left\{y,Ay,A^2 y\right\}$. Otherwise, there exists a functional $f\in{\X}'$ such that $f(y)=f(A^2y)=0$, $f(Ay)=2$ and $f(By)\neq2$ or $f(B^2y)\neq0$ to have $A\circ y \otimes f\in {\ixa}$ and  $B\circ y \otimes f\notin {\ixa}$ by Lemma \ref{f(x)=0}, a contradiction. Therefore, $By= \alpha y + \beta Ay +\gamma A^2 y$ for some $\alpha,\beta,\gamma\in\CC$ and $B^2y= \delta y + \epsilon Ay +\theta A^2 y$ for some $\delta,\epsilon,\theta\in\CC$. Since one can choose a functional $g\in{\X}'$ such that $g(y)=g(A^2y)=0$ and $g(Ay)=2$ we get $g(By)=2\beta=2$  and $g(B^2y)=2\epsilon=0$, which further implies that $\beta=1$  and $\epsilon=0$ confirming \eqref{eq:12}.
	
To prove (ii), choose $y_1=x$, $y_2=Ax$ and $y_3=x+Ax$, successively and apply \eqref{eq:12} to obtain 
\begin{align}
    Bx &= \alpha_1 x + Ax +\gamma_1 A^2 x, \label{eq:101} \\
    BAx &= \alpha_2 Ax + A^2x + \gamma_2 A^3 x, \label{eq:102} \\
    B(x+Ax) &= \alpha_3 (x+Ax) +(Ax+ A^2x) +\gamma_3 (A^2x + A^3 x), \label{eq:103}
\end{align}
for some $\alpha_i,\gamma_i\in\CC$, $i=1,2,3$. Adding up \eqref{eq:101} and \eqref{eq:102}, and then comparing the coefficients with those in \eqref{eq:103}, provides $\alpha_1=\alpha_2=\alpha_3$ and $\gamma_1=\gamma_2=\gamma_3$, confirming \eqref{eq:23}. That $BAx=ABx$ is now clear.
	
For (iii), as $\left\{x,Ax,A^2x,A^3x\right\}$ and $\left\{Ax,A^2x,A^3x,A^4x\right\}$ met the conditions of (ii) we have
\begin{equation}\label{Novo}
	\begin{array}{l}
		\ \   Bx=\alpha_1x+Ax+\gamma_1A^2x\\
			BAx=\alpha_1Ax+A^2x+\gamma_1A^3x
	\end{array}
		\ \ \text{and}\ \ 
	\begin{array}{l}
			\ BAx=\alpha_2Ax+A^2x+\gamma_2A^3x\\
		BA^2x=\alpha_2A^2x+A^3x+\gamma_2A^4x\end{array}
\end{equation}
for some $\alpha_1,\alpha_2,\gamma_1,\gamma_2\in\CC$. Obviously, $\alpha_1=\alpha_2=:\alpha$ and $\gamma_1=\gamma_2=:\gamma$. Our goal is to show that $\alpha=0$ and $\gamma=0$. To do this, at first observe that $B^2x \in {\s}\left\{x,A^2x\right\}$ from (i). On the other hand, acting by $B$ on $Bx=\alpha x+Ax+\gamma A^2x$ and applying (\ref{Novo}) gives
\begin{align*}
    B^2x &= \alpha Bx +BAx +\gamma BA^2x\\
    & = \alpha(\alpha x +Ax +\gamma A^2 x)+(\alpha Ax + A^2x + \gamma A^3 x) + \gamma(\alpha A^2x +A^3x +\gamma A^4 x)\\
    & = \alpha^2x + 2\alpha Ax +(1+2\alpha\gamma)A^2x + 2\gamma A^3x +\gamma^2 A^4 x.
\end{align*}
Hence, $\alpha=\gamma=0$ and in turn,
$Bx=Ax$, $BAx=A^2x$ and, $BA^2x=A^3x$ according to (\ref{Novo}).
\qed\end{proof}

\begin{Lemma}\label{lem:Jn(0)}
Let ${\X}={\mathcal Y}\oplus\mathcal Z$, where $\mathcal Y$ is a $n$-dimensional subspace of $\X$ with $1\leq n<\dim\X$. Suppose that $A,B\in\bx$ are algebraic operators such that
$$
A \circ X\in  {\mathcal I}^*({\X})\Leftrightarrow  B \circ X \in {\mathcal I}^*({\X})
\qquad\text{and}\qquad
A\circ X=0\Leftrightarrow B\circ X=0,
$$
for every $X \in {\mathcal F}_1(\X)\cup{\mathcal K}(\X)$. If with respect to the decomposition $\X={\mathcal Y}\oplus\mathcal Z$ we can represent operator $A$  as
\begin{equation}\label{eq:AinJ(0)}
A=\begin{bmatrix}
J_n(0)&A_{12}\\
0&A_{22}
\end{bmatrix},    
\end{equation}
then $By=Ay$ for every $y\in \mathcal Y$.	
\end{Lemma}

\begin{proof}
Let us denote by $e_1,e_2,\ldots,e_n$ the basis of $\mathcal Y$ in which the operator $A$ is represented as in (\ref{eq:AinJ(0)}). We then write $B$ in matrix for as
$
B=\left[\begin{smallmatrix}
    B_{11} & B_{12} \\
    B_{21}& B_{22}
\end{smallmatrix}\right].
$
Our goal is to show that $B_{11}=J_n(0)$ and $B_{21}=0$. Before proving this, observe that if $B_{21}=0$, then $B_{11}$ must be nilpotent; otherwise, we arrive at a contradiction with Lemma \ref{corr:alg}. We now begin with the proof.

For $n=1$, the claim follows directly from Corollary \ref{corr:alg}. So, let further $n=2$. Since $e_2$ and $Ae_2$ are linearly independent and $A^2e_2=0$, property (i) of Proposition \ref{prop:1234} ensures that $Be_2 =\alpha_0 e_2 +Ae_2$ for some scalar $\alpha_0$, which implies that $B_{21}=0$. Since $B_{11}$ is now nilpotent, we conclude that $\alpha_0=0$, and hence, $Be_2=Ae_2$. Combined with $Be_1=Ae_1$ we are done.

Next, consider the case $n\geq4$. We firstly use the fact that $\{e_n,Ae_n,A^2e_n,A^3e_n\}$ is linearly independent set and $A^4e_n=0$ or the set $\{e_n,Ae_n,A^2e_n,A^3e_n,A^4e_n\}$ is linearly independent. By property (iii) of Proposition \ref{prop:1234} we then have $Be_n=Ae_n$. Moreover, $\{e_j+e_n, A(e_j+e_n), A^2(e_j+e_n), A^3(e_j+e_n), A^4(e_j+e_n)\}$ is a linearly independent set for all $j=1,2,\ldots,n-1$. Thus, by the same property, $A(e_j+e_n)=B(e_j+e_n)$, and hence $Ae_j=Be_j$ for all $j=1,2,\ldots,n-1$, so that $Ay=By$ for every $y\in \mathcal{Y}$.

It remains to consider the case when $n=3$. In this case, we have $Ae_3=e_2$, $A^2e_3=e_1$ and $A^3e_3=0$. By property (ii) of  Proposition \ref{prop:1234} we find that $Be_3=\alpha e_3 +Ae_3 +\gamma A^2e_3$ and $BAe_3=\alpha Ae_3 + A^2e_3$ which can be simplified to
$$
Be_3=\alpha e_3+e_2+\gamma e_1
\qquad\text{and}\qquad
Be_2=\alpha e_2+e_1
$$
for some complex numbers $\alpha$ and $\gamma$. From Corollary \ref{corr:alg}  we also have $Be_1=0$ since $Ae_1=0$. Therefore, $B_{21}=0$ and since $B_{11}$ must be nilpotent, it follows that $\alpha=0$. To sum up,
\begin{equation}\label{eq:BinJ(0)}
B=\begin{bmatrix}
B_{11}&B_{12}&\\
0&B_{22}
\end{bmatrix},
\qquad\text{where}\qquad  
B_{11}=\begin{bmatrix}
   0&1&\gamma \\
   0&0&1\\
   0&0&0
\end{bmatrix}.
\end{equation}
We want to see that $\gamma=0$. To reach a contradiction, let us assume on the contrary that $\gamma\neq0$. 

Let us firstly suppose that $A^3\neq 0$.  If there exists a nonzero $\lambda\in\CC$ in the spectrum of $A$, with corresponding eigenvector $z$, we have $Az=\lambda z$ and also, by Corollary \ref{corr:J1}, $Bz=\lambda z=Az$. As $z\notin\{e_1,e_2,e_3\}$, the set 
$$
\{e_3+z,A(e_3+z),A^2(e_3+z),A^3(e_3+z)\}=\{e_3+z,e_2+\lambda z,e_1+\lambda^2 z,\lambda^3 z\}
$$
is linearly independent. 
Then, by property (ii) of Proposition \ref{prop:1234}, we have 
$$
B(e_3+z)=\alpha_1(e_3+z)+A(e_3+z)+\gamma_1 A^2(e_3+z),
$$ 
for some $\alpha_1,\gamma_1\in\CC$. Substituting $Be_3=e_2+\gamma e_1$, $Ae_3=e_2$, $A^2e_3=e_1$ and $Bz=\lambda z$, we obtain that $0=(\gamma_1-\gamma)e_1+\alpha_1 e_3+(\alpha_1+\gamma_1\lambda^2)z$. As $e_1,e_3,z$ are linearly independent, it follows that $\gamma=0$, a contradiction. Therefore, the spectrum of $A$ equals $\{0\}$, and since $A$ is algebraic, $A$ is nilpotent of nilindex $r\geq4$. Hence, there exists a linearly independent four-tuple $w, Aw, A^2w, A^3w$ with $A^4w=0$. We can take $w=A^{r-4}u$, for an existing linearly independent set $\{u, Au,\dots, A^{r-1}u\}\subseteq \bx$. Then, by (iii) of Proposition \ref{prop:1234}, we have
\begin{equation}\label{prop:Bw=Aw}
Bw=Aw,\qquad BAw=A^2w\qquad\text{and}\qquad BA^2w=A^3w.
\end{equation}
Suppose that for some $y\in \mathcal{Y}$ we have
$
y=\beta_0 w +\beta_1 A w+ \beta_2 A^2 w+\beta_3 A^3w,
$
for some $\beta_0,\beta_1,\beta_2,\beta_3\in\CC$. Then
\begin{align*}
    Ay&=\beta_0 Aw +\beta_1 A^2 w+ \beta_2 A^3 w,\\
    By&=\beta_0 Bw +\beta_1 BA w+ \beta_2 BA^2 w+\beta_3BA^3w.
\end{align*}
Inserting properties from (\ref{prop:Bw=Aw}) into linear combination of $By$ we get  $By=Ay+\beta_3 BA^3 w$. But, by Corollary \ref{corr:alg}, we have $\ker A=\ker B$ and, since $A(A^3w)=A^4w=0$ we see that $A^3w\in \ker A$. Consequently, $A^3w\in\ker B$ and hence, $By=Ay$. However, as $\gamma\neq0$, it is evident that $Ae_3\neq Be_3$ and so, $ {\s}\{w,Aw,A^2w,A^3w\}\cap\mathcal{Y}=\{0\}$. Now, consider the vector $e_3+w$. Clearly, $A^4(e_3+w)=0$ and $e_3+w, A(e_3+w),A^2(e_3+w),A^4(e_3+w)$ are linearly independent vectors. Thus,  by (iii) of Proposition \ref{prop:1234}, it follows that $A(e_3+w)=B(e_3+w)$, a contradiction since $Aw=Bw$ and $Ae_3\neq Be_3$. 

Therefore, $A^3=0$. If $\dim\X=3$, take the matrix
$$
D=\begin{bmatrix}
1&0&0 \\
0&-1&0\\
0&0&1
\end{bmatrix}\in{\mathcal K}(\X),
$$ 
so that $A\circ D=0$ while $B\circ D\neq0$, again a contradiction. Further, we may suppose that $\dim \X>3$. Let us write $A_{12}$ introduced in (\ref{eq:AinJ(0)}) as $A_{12}=e_1\otimes f +e_2\otimes g+e_3\otimes h$ for some bounded functionals $f,g,h$ on $\X$ satisfying $f(\mathcal{Y})=g(\mathcal{Y})=h(\mathcal{Y})=0$. It is an elementary exercise to check that $A^3=0$ if and only if $A_{22}^3=0$ and $h+A'g+A'^2f=0$ (the latter identity multiplied by $A'$ and $A'^2$, respectively, also gives $A'h+A'^2g=A'^2h=0$ as $ A'^3=0$). Let us define an operator $K\in \bx$ by
$$K=\begin{bmatrix}
    D & X\\
    0 & 0
\end{bmatrix},\qquad\text{where } X=-e_2\otimes f + e_3 \otimes (g+A'f).$$
It is easy to check that $K^3=K$, so $K\in \mathcal{K}(\X)$. Now, by applying \eqref{eq:BinJ(0)} we obtain
$A\circ K =0$
while $(B\circ K) e_3\neq0$, a final contradiction.

As a result we get that $\gamma=0$ in (\ref{eq:BinJ(0)}) and thus, $B_{11}=J_3(0)$, providing that $Ay=By$ for all $y\in\mathcal Y$.
\qed\end{proof}

We close this section with a characterization of when two operators coincide in terms of the Jordan product $\circ$. Let 
$$
{\mathcal T}= \mathcal{N}_1(\X) \cup \mathcal{K}(\X) 
\bigcup_{\lambda \in \mathbb{C}^\ast} 
\left\{
\begin{array}{ll}
\mathcal{T}(3;\lambda) & \text{if } \dim{\X}\geq4\\
\mathcal{T}(2;\lambda) & \text{if }  \dim \X = 3
\end{array}
\right. ,
$$
and note that $\mathcal T$ is  a similarity-invariant set, i.e. for every invertible $S\in \bx$ we have that $STS^{-1}\in \mathcal T$ for every $T\in \mathcal T$.

\begin{Lemma}\label{lem:operator}
Let $A,B\in \bx$. Then the following statements are equivalent.
\begin{enumerate}
	\item[(i)] $A=B$
	\item[(ii)] $A\circ X\in {\ix} \Leftrightarrow B\circ X\in {\ix}$, for every $X\in \bx$.
	\item[(iii)] $A\circ X\in {\ix} \Leftrightarrow B\circ X\in {\ix}$, for every $X\in \mathcal T$.
    
\end{enumerate}
\end{Lemma}
		
\begin{proof}
The implications (i)$\Rightarrow$(ii)$\Rightarrow$(iii) are straightforward. So, assume that (iii) holds true and proceed to show that this implies $A=B$. By Observation \ref{Observation} note that
\begin{equation}\label{eq:relation}
A\circ X\in{\ixa}\Leftrightarrow B\circ X\in {\ixa}
\quad\text{and}\quad
A\circ X=0\Leftrightarrow B\circ X=0,
\end{equation}
for every $X \in  \mathcal T$. In the following, we will distinguish two cases, depending on whether $A$ and $B$ are algebraic.

At first, we suppose that $A$ is not algebraic. Choose an arbitrary $x\in \mathcal{X}$. If  the set $\{x,Ax,A^2x,A^3x,A^4x\}$ is linearly independent, then $Ax=Bx$ by (iii) of Proposition \ref{prop:1234}. If this is not the case, let $\mathcal{X}=\mathrm{span}\{x,Ax,A^2x,A^3x,A^4x\}\oplus \mathcal{Z}$ and write 
$$A=\begin{bmatrix}
    A_{11} & A_{12} \\
    0 & A_{22}
    \end{bmatrix}
$$
with respect to this decomposition. Note that $A_{22}$ is not algebraic. So, there exists a $z\in \mathcal{Z}$ such that $z,A_{22}z,A_{22}^2z,A_{22}^3z,A_{22}^4z$ are linearly independent. It is then easy to see that both $\{z,Az,A^2z,A^3z,A^4z\}$ and $\{x+z,A(x+z),A^2(x+z),A^3(x+z),A^4(x+z)\}$ are linearly independent as well. By (iii) of Proposition \ref{prop:1234} it follows that $Az=Bz$ and $A(x+z)=B(x+z)$ providing $Ax=Bx$.

The same conclusion follows if $B$ is not algebraic.

Suppose further that both $A$ and $B$ are algebraic operators. To reach a contradiction, assume that $A\neq B$. Then there exists a vector $x\in\X$ such that $Ax \neq Bx$. As at least one of $Ax$ and $Bx$ is nonzero, we may and we do suppose that $Ax\neq0$.

If $x$, $Ax,\dots$, $A^4x$ are linearly independent or, $x$, $Ax$, $A^2x$, $A^3x$ are linearly independent and $A^4x=0$,  we arrive at a contradiction by (iii) of Proposition \ref{prop:1234}. Hence, for
$$
E_A(x) = {\s}\left\{x,Ax,A^2x,A^3x,A^4x\right\}
$$
we have  $1\leq m:=\dim E_A(x)\le 4$.   It is easy to see that $\left\{x,Ax,\ldots, A^{m-1}x\right\}$ is a basis of $E_A(x)$ and  $E_A(x)$ is a $m$-dimensional  $A$-invariant subspace of $\mathcal{X}$. Let us split $\X=E_A(x)\oplus \mathcal{Z}$. With respect to this decomposition, the operator $A$ can be represented by
$$
A=\begin{bmatrix}
    A_{11} & \ast \\
    0 & \ast
\end{bmatrix},
$$
where $A_{11}$ is a $m\times m$ complex matrix (recall that $m=\dim E_A(x)$). There is no loss of generality in assuming that $A_{11}$ is in Jordan normal form. Because the vectors $x, Ax, \dots, A^{m-1}x$ are linearly independent, the minimal polynomial of $A_{11}$ is equal to its characteristic polynomial. Thus, $A_{11}$ is nonderogatory, i.e. it has exactly one Jordan block per eigenvalue, and all eigenvalues are distinct.  Now, since $A_{11}$ is a direct sum of distinct Jordan blocks $J_{n_k}(\lambda_k)$, where $1\le n_k\le m\le4$, apply Lemma \ref{lem:Jn(lambda)} if $\lambda_k\neq0$ and  Lemma \ref{lem:Jn(0)} otherwise, to finally conclude that $Ay=By$ for every $y\in E_A(x)$. But then also $Ax=Bx$, a contradiction. Therefore, $A=B$ as desired. 
\end{proof}\qed

%%%%%%%%%%%%%%%%%%%%%%%%%%%%%%%%%%%%5
\section{Proofs of the main results}

%%%%%%%%%%%%%%%%%%%%%%%%%%%%%%
%%%%%%%%%%%%%%%%%% P R O O F S
%%%%%%%%%%%%%%%%%%%%%%%%%%%%%%
		
We are now ready to start simultaneously proving Theorems \ref{main_result} and \ref{main_result_matrix}.\\
		
\noindent\textbf{Proof of Theorems \ref{main_result} and \ref{main_result_matrix}.} Let us assume that $\phi:{\bx}\rightarrow\bx$ is a  map such that   ${\mathcal I}^\pm(\X)$ is a subset of the range of $\phi$. Further suppose that $A\circ B$ is idempotent if and only if $\phi(A)\circ\phi(B)$ is idempotent for any  $A$ and $B\in\bx$. The proof will be set up through several steps.\\

%%%%%%% STEP 1 : \phi(0)=0 ZADOŠČA, DA SLIKA VSEBUJE I_1 IN -I_1.
{\sc{Step 1.}} $\phi(0)=0$.

The claim is a direct application of Lemma \ref{Lemma:A=0} and the fact that the range of $\phi$ contains $\mathcal{I}_1^\pm(\X)$. \\ 
		
%%%%%%%%%%%%% 

{\sc Step 2}. \textit{$\phi$ is injective.} 

Let $\phi(A)=\phi(B)$. Obviously, $\phi(A)\circ Y \in \ix$ if and only if $  \phi(B)\circ Y \in \ix$ for every $Y$ in the range of $\phi$. It follows that $A\circ X  \in \ix$ if and only if $B\circ X \in \ix$ for every $X\in \bx$. By Lemma \ref{lem:operator}, $A=B$. \\

%%%%%%% STEP 2: identiteta gre v identiteto
{\sc{Step 3.}} \textit{$\phi(I)=I$ or $\phi(-I)=I$.}

As $I$ is in the range of $\phi$, there exists an $A$ such that $\phi(A)=I$. Clearly, $A^2=A\circ A$ is idempotent and by {\sc Step 1} and {\sc Step 2}, $A\neq 0$. It follows that either $A^2=I$ or  $A^2$ has a nontrivial kernel. Accordingly, we separate two cases. In each  we find an operator $B$   such that $A\circ B \in \mathcal{I}(X)$ and $B\circ B=B^2 \notin \mathcal{I}(X)$. Then $I\circ \phi(B)=\phi(B) \in \mathcal{I}(X)$ and, in turn, also $\phi(B)^2=\phi(B)\circ \phi(B)$ should be idempotent. But this contradicts the fact that $B^2$ is not idempotent.

{\sc Case 1. }  $A^2=I$. If $A$ is a scalar operator, we are done. Otherwise, $A$ is a nonscalar involution. 	Applying Lemma \ref{Lemma:id} provides a $B$ with the required properties to reach a contradiction.

{\sc Case 2. } $A^2=P$, where $P$ is an  idempotent with a nontrivial kernel (possibly $P=0$). It is easy to verify that both $\ker P$ and $\im P$ are $A$-invariant subspaces.  Define the  operator $N:\ker P\to \ker P$ as the restriction of $A$ to $\ker P$ by $Nx:=Ax$ for all $x\in \ker P$. Clearly, $N^2=0$. If   $N$ is nonzero, then $\dim \ker P >1$ and, we apply Lemma \ref{Lemma:id} to find a $B_0\in {\B}(\ker P)$ such that $N\circ B_0 \in \mathcal{I}(\ker P)$ and $B_0\circ B_0 \notin \mathcal{I}(\ker P)$. Then $B=B_0 \oplus 0_{\im P}$ will do the job. Finally, if the restriction of $A$ to $\ker P$ is equal to zero, then take any nonzero $x\in \ker P$, split $\ker P= \CC  x \oplus \mathcal{Y}$ and choose an $f\in \X'$ such that $f(x)=2$ and $f(\mathcal{Y}\oplus \im P)=0$. Now the  operator $B=x\otimes f$ has the required property.\\

%%%%%%%%%%%%%%%%%%%%%%%%%	
From now on, without loss of generality, we assume that $\phi(I)=I$. The map $X\mapsto -\phi(X)$ has the same properties as $\phi$, so by {\sc Step 3}, we get that $\phi(-I)=-I$. Consequently, $\phi({\ix})\subseteq\ix$ and $\phi(-{\ix})\subseteq-\ix$. Furthermore, observe that for every idempotent (anti-idempotent) $P$ in the range of $\phi$, its pre-image $\phi^{-1}(P)$ is an idempotent (anti-idempotent, resp.).\\

%%%%%%%% STEP: aditivnost med idempotentom in identiteto
{\sc Step 4.} \textit{For every $P\in{\I}_1(\X)$ we have $\phi(P-I)=\phi(P)-I$.}
		
As $\phi(-{\ix})\subseteq-\ix$, for an arbitrary rank-one idempotent $P$, we conclude that $\phi(P-I)=Q-I$ for some nontrivial $Q\in\ix$.  Since $Q$ is in the range of $\phi$, $\phi^{-1}(Q)$ is idempotent by the above observation. Hence, from $Q\circ (Q-I)=0$ we know that $\phi^{-1}(Q) \circ (P-I)$ is idempotent. More precisely, $\phi^{-1}(Q)\circ(P-I)=0$ by Lemma \ref{Lemma:PQ}. Now, as $ \rk P=1$ and  $\phi^{-1}(Q)$ is a nonzero idempotent, applying Lemma \ref{Lemma:PQ=0} gives that $\phi^{-1}(Q)=P$ and so, $\phi(P-I)=\phi(P)-I$. \\

%%%%%%%% STEP: ohranjanje ranga 1 na idempotentih
{\sc Step 5.} \textit{ $P\in\mathcal{I}_1(\X)$ if and only if $\phi(P)\in\mathcal{I}_1(\X)$.}
		
Let us take any $P\in \mathcal{I}_1(X)$ and note that $\phi(P)$ is a nontrivial idempotent. As in the range of $\phi$ there are all idempotents of rank one, we can choose an $R\in\mathcal{I}_1(\X)$ such that $R\circ (\phi(P)-I)=0$. The previous step gives us $ R\circ\phi(P-I)=0$ and in turn, by Lemma \ref{Lemma:PQ}, we obtain that $\phi^{-1}(R)\circ (P-I)=0$. Hence, as $\rk P=1$ and $\phi^{-1}(R)$ is a nonzero idempotent, Lemma \ref{Lemma:PQ=0} provides that $\phi^{-1}(R)=P$ and therefore,  $\rk \phi(P)=\rk R=1$. 
		
For the reverse implication, let $\phi(P) \in \mathcal{I}_1(\X)$. By the observation made immediately prior to {\sc Step 4}, it follows that $P$ is an idempotent; moreover, due to the injectivity of $\phi$, this preimage is unique and nontrivial. To show that $\rk P = 1$, suppose to the contrary that $\rk P > 1$. In that case, there exists a rank-one idempotent $S \neq P$ such that $S \circ (P - I) = 0$. Applying Step 4, we conclude that $\phi(S)\circ \phi(P-I)=\phi(S)\circ(\phi(P)-I)\in\ix$. Then, by Lemma \ref{Lemma:PQ}, it follows that $\phi(S)\circ(\phi(P)-I)=0$. Since $\phi(P)$ is a rank-one idempotent and $\phi(S)$ is a nonzero idempotent, Lemma \ref{Lemma:PQ=0} implies that $\phi(S) = \phi(P)$, contradicting the injectivity of $\phi$. Therefore, $\rk P = 1$, and hence $P \in \mathcal{I}_1(\X)$.\\

%%%%%%%% STEP: -P
{\sc Step 6.} \textit{For every $P\in{\I}_1(\X)$ we have $\phi(-P)=-\phi(P)$.}
		
The map $X\mapsto -\phi(-X)$, $X\in \bx$, inherits the same preserving properties as $\phi$. So, applying {\sc Step 5} for this particular map, $-\phi(-P)$ is a rank-one idempotent for every $P\in \I_1(\X)$.
				
Let us write $\phi(-P)=-Q$ for some $Q\in{\I}_1(\X)$ and we want to show that $Q=\phi(P)$. Using {\sc Step 4} and the identity $-P\circ (P-I)=0$ we obtain $\phi(-P)\circ\phi(P-I)=-Q\circ \left(\phi(P)-I\right)=Q\circ\left(I-\phi \left(P\right)\right)\in\ix$.  If $Q\neq\phi(P)$, by Lemma \ref{Lemma:-P} there exists an $R\in\I_1(\X)$ (and hence in the range of $\phi$) such that $Q\circ R=0$ and $\phi(P)\circ R\notin \ix$. From $Q\circ R=0$ it follows that also $-Q\circ R=0$ and therefore, $-P\circ\phi^{-1}(R)\in\ix$. Since $\phi^{-1}(R)$ is an idempotent (by the Observation preceding {\sc Step 4}), Lemma \ref{Lemma:PQ} provides that  $-P\circ \phi^{-1}(R)=0$. So,  $P\circ\phi^{-1}(R)=0$ while $\phi(P)\circ R\notin \ix$, a contradiction. Therefore, $Q=\phi(P)$ as desired.\\ 

%%%%%%%%%%%%%%%%%%
Let $\psi:\mathcal{I}_1(\X)\rightarrow\mathcal{I}_1(\X)$ with $\psi(P)=\phi(P)$ denote the restriction of $\phi$ to the set of rank-one idempotents. Note that the assertions established in {\sc Steps 4}, {\sc 5} and {\sc 6} are valid also for the map $\psi$.\\
		
%%%%%%% STEP : ohranjanje ortogonalnosti na idempotentih ranga 1
{\sc Step 7.} \textit{The map $\psi$ is bijective and preserves orthogonality in both directions, i.e. $P\bot Q$ if and only if $\psi(P)\bot \psi(Q)$, $P,Q\in{\mathcal{I}}_1(\X)$.} 
		
It is clear that $\psi$ is surjective by {\sc Step 5} and its injectivity is due to the injectivity of $\phi$. For any $P,Q\in{\I}_1(\X)$ with $P\bot Q$ we have $-P\circ Q=0$. Applying {\sc Step 6} we obtain $\phi(-P)\circ\phi(Q)=-\phi(P)\circ\phi(Q)\in \ix$. Using Lemma \ref{Lemma:PQ} we then conclude that $\phi(P)\circ\phi(Q)=0$, which shows that  $\phi(P)\bot\phi(Q)$. 
		
The proof of {\sc Step 7} can be closed by applying similar properties of $\psi^{-1}$.\\

Now, if $\X$ is infinite-dimensional the conditions of \cite[Theorem 2.4]{Semrl_commutativity} are met, and so,  there exists a bounded invertible linear or conjugate-linear operator $T:\X\rightarrow\X$ such that $\psi(P)= TPT^{-1}$, for $P\in \I_1(\X)$, or, $\psi(P)=TP^\prime T^{-1}$ for some bounded invertible linear or conjugate-linear operator $T:{\X}'\rightarrow\X$. Passing to the map $X\mapsto T^{-1}\phi(X)T$, $X\in \bx$, in the first case, and to the map $X \mapsto T^\prime \phi(X)^\prime {T^\prime}^{-1}$ otherwise (and identifying $\X^{\prime\prime}=\X$), where each of the new maps has the same properties as given $\phi$, we suppose that $\phi(P)=P$ for all $P\in \I_1(\X)$. Similarly, if $\X$ is finite-dimensional, properties of \cite[Theorem 2.3]{Semrl_commutativity} are fulfilled for both $\psi$ and $\psi^{-1}$. Hence, also in this case we can, without loss of any generality, assume that $\phi(P)=P$ for every $P\in\I_1(\X)$.
		
Having this information at hand, consider the map $X\mapsto-\phi(-X)$, possessing the same preserving properties as $\phi$. At first, we observe that it fixes all rank-one idempotents; this is due to {\sc Step 6} applied to $\phi$, and then that $\phi(-P)=-P$ for all $P\in\mathcal{I}_1(\X)$. To sum up, $\phi(P)=P$, for every $P\in\mathcal{I}_1^\pm(\X)$.\\

%%%%%%%%%%%

{\sc Step 8.} \textit{For every $A\in\bx$ and every $P\in\mathcal{I}_1^\pm(\X)$ we have $A\circ P\in{\ixa}\Leftrightarrow\phi(A)\circ P\in\ixa$ and $A\circ P=0\Leftrightarrow\phi(A)\circ P=0$.}

The claim is a direct consequence of Observation \ref{Observation}. Indeed, let us suppose that there exists an $A \in \mathcal{B}(\X)$ and $P \in \mathcal{I}_1^\pm(\X)$ such that $A \circ P\in\ixa$ and $\phi(A)\circ P=0$. By Observation \ref{Observation} we then obtain that $A\circ(-P)\notin\ix$ and so, $\phi(A)\circ(-P)\notin\ix$, leading to a contradiction with $\phi(A)\circ P=0$.\\	

%%%%%%% STEP: phi(P-Q)=P-Q za ortogonalna P,Q ranga 
{\sc Step 9.} {\textit{$\phi(K)=K$ for every $K\in{\mathcal K}(\X)$. Consequently, for every $A\in\bx$ and $K\in{\mathcal K}(\X)$ we have $A\circ K\in\ixa\Leftrightarrow \phi(A)\circ K\in\ixa$ and $A\circ K=0\Leftrightarrow \phi(A)\circ K=0$.}}

Applying {\sc Step 8}, the first part of the claim is an immediate consequence of Corollary \ref{corr:D}. As for the second part, observe that $K \in \mathcal{K}(\X)$ if and only if $-K \in \mathcal{K}(\X)$. Hence, we have $\phi(K) = K$ and $\phi(-K) = -K$ for every such operator $K$. The conclusion then follows directly from Observation \ref{Observation}.\\

%%%%%%%% STEP: \phi(A)=A za nilpotente
{\sc Step 10.} \textit{$\phi(N)=N$ for every $N\in{\mathcal N}_1(\X)$. Consequently, for every $A\in\bx$ and $N\in{\mathcal N}_1(\X)$ we then have $A\circ N\in{\ixa}\Leftrightarrow\phi(A)\circ N\in\ixa$ and $A\circ N=0\Leftrightarrow\phi(A)\circ N=0$.}

For an arbitrary $N\in\mathcal{N}_1(\X)$, we can take a nonzero $x\in\X$ and $g\in{\X}'$ with $g(x)=0$ such that $N=x\otimes g$. Then, take an $y\in\X$ with $g(y)=1$ and, since $x$ and $y$ are linearly independent, we can choose a functional $f\in{\X}'$ such that $f(x)=1$ and $f(y)=0$. The operators $P:=x\otimes f$ and $Q:=y\otimes g$ are now orthogonal rank-one idempotents, and hence, $R:=P+Q$ is a nontrivial idempotent. From $N\circ(I-R)=0$, {\sc Step 9} gives that  $\phi(N)\circ(I-R)=0$. Furthermore, Lemma \ref{Lemma:PQ=0} ensures that $\phi(N)x=0$ for every $x\in\ker R$ and that $\im R$ is $\phi(N)$-invariant subspace. 

Hence, according to the decomposition ${\mathcal X}=\im R\oplus\ker R$ and choosing an appropriate basis in $\im R$, the operators $N$ and $\phi(N)$ have the following matrix representations:
$$
N=\begin{bmatrix}0&1\\0&0\end{bmatrix}\oplus0
\qquad\text{and}\qquad
\phi(N)=\begin{bmatrix}a&b\\c&d\end{bmatrix}\oplus0.
$$
Now, for every $\alpha,\beta\in\CC$, we can introduce the family of rank-two tripotent operators 
$$
K_{\alpha}=\begin{bmatrix}1&\alpha\\0&-1\end{bmatrix}\oplus0
\qquad\text{and}\qquad
H_{\beta}=\begin{bmatrix}1&0\\\beta&-1\end{bmatrix}\oplus0.
$$
Since $N\circ K_{\alpha}=0$ for every $\alpha$, and $N\circ H_{2}\in{\mathcal I}^*(\X)$, applying {\sc Step 9}  gives that
$$
\phi(N)\circ \phi(K_{\alpha})=\phi(N)\circ K_{\alpha}=
\begin{bmatrix}
a+\frac{c}{2}\alpha & \frac{(a+d)}{2}\alpha\\
0&-d+\frac{c}{2}\alpha
\end{bmatrix}\oplus0=0
$$
and
$$
\phi(N)\circ \phi(H_{2})= \phi(N)\circ H_{2}=
\begin{bmatrix}
a+b & 0\\
a+d&b-d
\end{bmatrix}\oplus0\in{\mathcal I}^*(\X),
$$
for every $\alpha\in\CC$. From the first equation, we obtain $a=c=d=0$, and from the second property, we conclude that $b=1$. Thus, $\phi(N)=N$ as desired.

As $\phi(N)=N$ and $\phi(-N)=-N$ for every $N\in\mathcal{N}_1(\X)$, the consequence follows immediately from Observation \ref{Observation}.\\

%%%%%% STEP: \phi(A)=A za trace-zero operatorje ranga 2
{\sc Step 11.} \textit{$\phi(\lambda(P-Q))= \lambda(P-Q)$ for every orthogonal $P,Q\in{\mathcal I}_1(\X)$ and every nonzero  $\lambda\in\CC$. }

For arbitrary orthogonal $P,Q\in {\mathcal I}_1(\X)$ set  $K:=\lambda(P-Q)$ and $R:=P+Q\in\ix$. Since $K\circ(I-R)=0$, it follows that $\phi(K)\circ(I-R)=0$ by {\sc Step 9}. Applying Lemma \ref{Lemma:PQ=0} we conclude that $Kx=0=\phi(K)x$ for every $x\in\ker R$ and that $\im R$  is $\phi(K)$-invariant subspace. With respect to the  decomposition ${\X}=\im R\oplus\ker R$ and choosing a suitable basis in $\im R$,  we can write $K$ and $\phi(K)$ in the following matrix forms:
$$ 
K=\begin{bmatrix}\lambda&0\\0&-\lambda\end{bmatrix}\oplus0
\qquad\text{and}\qquad
\phi(K)=\begin{bmatrix}a&b\\c&d\end{bmatrix}\oplus0.
$$
Next, for any $\alpha,\beta\in\CC$, consider the families of rank-one nilpotents
$$
N_\alpha=\begin{bmatrix}
0&\alpha\\0&0    
\end{bmatrix}\oplus0
\qquad\text{and}\qquad
M_\beta=\begin{bmatrix}
0&0\\\beta&0    
\end{bmatrix}\oplus0.
$$
As $K\circ N_{\alpha}=0$ and $K\circ M_{\beta}=0$ for all $\alpha,\beta\in\CC$, using {\sc Step 10} gives that
$$
\phi(K)\circ\phi(N_\alpha)=\phi(K)\circ N_{\alpha}=\frac12\begin{bmatrix}c\alpha&(a+d)\alpha\\0&c\alpha\end{bmatrix}\oplus0=0
$$
and
$$
\phi(K)\circ\phi(M_\beta)=\phi(K)\circ M_{\beta}=\frac12\begin{bmatrix}b\beta&0\\(a+d)\beta&b\beta\end{bmatrix}\oplus0=0,
$$
for all $\alpha,\beta\in\CC$. This implies that $c=b=0$ and $d=-a$.  In turn, $\phi(K) = \frac{a}{\lambda}K$. Let us then introduce a rank-one nilpotent 
$$N_\lambda = \begin{bmatrix}\lambda^{-1} & -1\\ \lambda^{-2} & -\lambda^{-1}     \end{bmatrix} \oplus 0. $$
Then $K\circ N_\lambda = I\oplus 0$, while $\frac{a}{\lambda}K\circ N_\lambda = \frac{a}{\lambda}I\oplus 0$, obviously forcing $a=\lambda $ since $a\neq 0$ by {\sc Step 1}. So, $\phi(K) = K$.\\

%%%%%% STEP 10: \phi(A)=A za vsak večkratnik idempotenta
{\sc Step 12.}  \textit{$\phi(\lambda P)=\lambda P$ for every $P\in{\mathcal I}_1(\X)$ and nonzero $\lambda\in\mathbb C$.}

As $\lambda P$ is now of rank one, Lemma \ref{Lemma:PQ=0} ensures that $\phi(\lambda P)=\mu P$ for some $\mu\in{\CC}^\ast$. Then, set $K:=\lambda^{-1}(P-Q)$, where $Q\in \mathcal{I}_1(\X)$ is orthogonal on $P$. Since $\lambda P\circ K=P\in\ixa$, we can conclude that $\phi(\lambda P)\circ\phi(K)=\mu P\circ K=\frac{\mu}{\lambda}P\in\ixa$ by {\sc Step 11}. Therefore, $\mu=\lambda$ and so,  $\phi(\lambda P)=\lambda P$.\\

%%%%%%%%%%%%%%%% "triangular" operators
{\sc Step 13.} \textit{Let $\lambda\in{\CC}^\ast$ be an arbitrary. Then $\phi(T)=T$ for every $T\in {\mathcal T}(k;\lambda)$, where $k<\dim\X$.}

We apply induction on $k$. Since the operator $T$ can be written as $T=J_k(\lambda)\oplus0$ with respect to some basis of $\X$, Lemma \ref{Lemma:PQ} provides that $\phi(T)=B_{11}\oplus0$ with respect to the same basis of $\X$. The aim is to see that $B_{11}=J_k(\lambda)$.

For $k=1$ the claim holds true by Corollary \ref{corr:J1}. Then, assuming the claim holds for some $k$, applying Lemma \ref{lem:Jn(lambda)} at the induction step completes the proof.\\

%%%%%%%%%%%%%%%%%%%%%%%%%%%%%%%%%%%
{\sc Step 14.} \textit{$\phi(A)=A$ for every $A\in \bx$.}

At first, let $A=\lambda I$ for some $\lambda\in{\CC}^\ast$. Hence, from $A\circ\lambda^{-1}P=P\in\ixa$ for every $P\in{\mathcal I}_1(\X)$, it follows that $\phi(A)\circ\lambda^{-1}P=\lambda^{-1}\phi(A)\circ P\in\ixa$ for every $P\in{\mathcal I}_1(\X)$. Lemma \ref{Lemma:A=I} then provides that $\phi(A)=\lambda I$.

Let us now assume that $A$ is not scalar operator and denote $B:=\phi(A)$. So far, we have shown that
$A\circ X\in{\ixa}\Leftrightarrow B\circ X\in{\ixa}$
and 
$A\circ X=0\Leftrightarrow B\circ X=0$,
for every $X\in\mathcal{T}$. Then Lemma \ref{lem:operator} finally closes the proof.
\qed\\

%%%%%%%%%%%%%%%%%%%%%%%%%%%%%%%%%%%%%%%%%%%%%%%%%%%
\noindent\textbf{Proof of Theorems \ref{main_result2} and \ref{main_result2_matrix}.} Define the map $\psi:{\B(\X)}\rightarrow\B(\X)$ by $\psi(X)=\sqrt{2\alpha}\;\phi\left(\frac{1}{\sqrt{2\alpha}}X\right)$, for every $X\in\B(\X)$. The map $\psi$ is well-defined and satisfies the assumptions of Theorem \ref{main_result2} in the infinite-dimensional setting. For the finite-dimensional case, we can assume without loss of any generality that the map $\phi$ is of the form $\psi([x_{ij}])=[\sigma(x_{ij})]$ by Theorem \ref{main_result2_matrix}. Then $$\phi([x_{ij}])=1/\sqrt{2\alpha}\,\psi([\sqrt{2\alpha}x_{ij}])=\sigma(\sqrt{2\alpha})/\sqrt{2\alpha}\, [\sigma(x_{ij})] =c_\alpha [\sigma(x_{ij})]. $$ 
Thus, the desired result follows.
\qed

\vskip 1pc \noindent{\bf Acknowledgments.} 
Researchers were partially supported by Slovenian research agency
ARIS, program no.\ P1-0288. 

%%%%%%%%%%%%%%%%%%%%%%%%%%%%%
%
% L I T E R A T U R A
%
%%%%%%%%%%%%%%%%%%%%%%%%%%%% 

%%%%%%%%%%%%%%%%%%%%%%%%%%%%%%%%%%%%%%%%%%%%%%

\end{document}